\documentclass[11pt,reqno,fleqn]{article}
\usepackage{amsfonts,amssymb,mathrsfs,amsmath,amsthm,relsize}
\usepackage[top=20mm,bottom=20mm,left=20mm,right=20mm]{geometry}
\usepackage{times}
\usepackage{multicol}
\usepackage{graphicx}
\usepackage{graphics}
\usepackage{enumitem}
\usepackage{cite}
\usepackage{mathtools}
\usepackage{etoolbox}
\usepackage{fancyhdr}
\fancypagestyle{titlepage}
{
   \fancyhf{}
   \lhead{The final publication is available at IOS Press through http://dx.doi.org/10.3233/JIFS-210596}
   \rhead{}
   \fancyfoot[C]{\thepage}
}

\DeclarePairedDelimiter\floor{\lfloor}{\rfloor}
\newcommand{\ccomma}{\mathbin{\raisebox{0.5ex}{,}}}
\allowdisplaybreaks
\makeatletter
\let\@fnsymbol\@arabic
\makeatother

\newcommand\blfootnote[1]{%
  \begingroup
  \renewcommand\thefootnote{}\footnote{#1}%
  \addtocounter{footnote}{-1}%
  \endgroup
}

\theoremstyle{definition}
\newtheorem{theorem}{Theorem}[section]
\numberwithin{equation}{section}
\newtheorem{definition}[theorem]{Definition}
\newtheorem{lemma}[theorem]{Lemma} %%Delete [teo] to re-start numbering
\newtheorem{example}[theorem]{Example}
\makeatletter
\patchcmd{\@maketitle}{\begin{center}}{\begin{flushleft}}{}{}
\patchcmd{\@maketitle}{\begin{tabular}[t]{c}}{\begin{tabular}[t]{@{}l}}{}{}
\patchcmd{\@maketitle}{\end{center}}{\end{flushleft}}{}{}
\makeatother
\begin{document}
\title{Tauberian theorems for statistical Ces\`{a}ro and statistical logarithmic summability of sequences in intuitionistic fuzzy normed spaces}
\author{Enes Yavuz}
\date{{\small Department of Mathematics, Manisa Celal Bayar University, Manisa, Turkey\\E-mail: enes.yavuz@cbu.edu.tr}}
\maketitle
\thispagestyle{titlepage}
\blfootnote{\rule{0.63cm}{0cm}\emph{\!\!Mathematics Subject Classification:} 03E72, 40A05, 40G05, 40G15, 40E05}

\noindent\textbf{Abstract:} We define statistical Ces\`{a}ro and statistical logarithmic summability methods of sequences in intuitionistic fuzzy normed spaces($IFNS$) and give slowly oscillating type and Hardy type Tauberian conditions under which statistical Ces\`{a}ro summability and statistical logarithmic summability imply convergence in $IFNS$. Besides, we obtain analogous results for the higher order summability methods as corollaries. Also, two theorems concerning the convergence of statically convergent sequences in $IFNS$ are proved in the paper.
\noindent
\section{Introduction and preliminaries}
Events and problems that humankind encounter in real world are commonly complex and imprecise due to uncertainty of parameters and indefiniteness of the objects involved. Researchers proposed various theories of uncertainty to handle such events and problems involving incomplete information. Among them, Lotfi A. Zadeh proposed fuzzy logic as an extension Boolean logic with the introduction of the mathematical concept of fuzzy sets\cite{zadeh}. In \cite{zadeh}, Zadeh extended classical sets to fuzzy sets by using gradual memberships taking real values in the interval $[0,1]$, instead of Boolean memberships which take only integer values $\{0,1\}$. Following its introduction, fuzzy sets were utilized by many researchers in different fields of science to process non-categorical data. Besides, motivated by fuzzy sets, Atanassov\cite{atanassov1983,atanassov} defined intuitionistic fuzzy sets($IFS$) by considering gradual non-memberships as well as Zadeh's gradual memberships. Like fuzzy sets, $IFS$ were also applied in many areas of science such as control systems, robotics, computer, medical diagnosis, education, etc. Theoretical basis of intuitionistic fuzzy set theory also developed in time. In particular, Park\cite{ifmetric} defined $IF-$metric and it was pursued by $IF-$norm\cite{ifnorm}. Convergence of sequences with respect to $IF-$norm was defined and various convergence methods in the statistical sense were proposed to achieve a limit where ordinary convergence fails\cite{karakus,mursaleen1,muhiddin,mursaleen3,jifs1,jifs2,jifs3,jifs4,selma,jifs5}. Also, there are recent studies applying some weighted mean summation methods in $IFNS$ to achieve limit of sequences and providing Tauberian conditions which guarantee ordinary convergence\cite{yavuzcesaro,yavuzlog}.

In some cases of sequences as in Example \ref{ex1}, Example \ref{ex2}, Example \ref{ex3} and Example \ref{ex4}, a limit can not be achieved via statistical convergence or via weighted mean summation methods in $IFNS$. In such cases, we need to use some stronger methods of convergence to achieve a limit in $IFNS$, and to investigate whether ordinary convergence can be guaranteed with some extra conditions. For this aim, in this study we introduce statistical Ces\`{a}ro and statistical logarithmic summability of sequences in $IFNS$, and obtain Tauberian conditions of slowly oscillating type and Hardy type under which ordinary convergence in $IFNS$ follows from statistical Ces\`{a}ro and statistical logarithmic summability. Besides, we show that statistical convergence of slowly oscillating sequences in $IFNS$ implies ordinary convergence in the space. Examples in the paper provide also new types of sequences for statistical Ces\`{a}ro and statistical logarithmic summation methods in classical normed spaces. Before to continue main results, we now give some preliminaries.
\begin{definition}\cite{lael}
The triplicate $(N,\mu,\nu)$ is said to be an $IFNS$ if $N$ is a real vector space, and $\mu,\nu$ are fuzzy sets on $N\times\mathbb{R}$ satisfying the following conditions for every $u,w\in N$ and $t,s\in\mathbb{R}$:
\begin{enumerate}[label=(\alph*)]
\item $\mu(u,t)=0$ for $t\leq0$,
\item $\mu(u,t)=1$ for all $t\in\mathbb{R^+}$ if and only if $u=\theta$
\item $\mu(cu,t)=\mu\left(u,\frac{t}{|c|}\right)$ for all $t\in\mathbb{R^+}$ and $c\neq0$,
\item $\mu(u+w,t+s)\geq \min\{\mu(u,t),\mu(w,s)\}$,
\item $\lim_{t\to\infty}\mu(u,t)=1$ and $\lim_{t\to0}\mu(u,t)=0$,

\item $\nu(u,t)=1$ for $t\leq0$,
\item $\nu(u,t)=0$ for all $t\in\mathbb{R^+}$ if and only if $u=\theta$
\item $\nu(cu,t)=\nu\left(u,\frac{t}{|c|}\right)$ for all $t\in\mathbb{R^+}$ and $c\neq0$,
\item $\max\{\nu(u,t),\nu(w,s)\}\geq\nu(u+w,t+s)$,
\item $\lim_{t\to\infty}\nu(u,t)=0$ and $\lim_{t\to0}\nu(u,t)=1$.
\end{enumerate}
We call $(\mu,\nu)$ an $IF-$norm on $N$.
\end{definition}
\begin{example}\label{standartnorm}
Let $(N, \Vert\cdot\Vert)$ be a normed space and $\mu_0$, $\nu_0$ be fuzzy sets on $N\times\mathbb{R}$ defined by
\begin{eqnarray*}
\mu_0(u,t)=
\begin{cases}
0, \quad &t\leq0,\\
\frac{t}{t+\Vert u\Vert}\ccomma &t>0,
\end{cases}
\hspace{2cm}
\nu_0(u,t)=
\begin{cases}
1, \quad &t\leq0,\\
\frac{\Vert u\Vert}{t+\Vert u\Vert}\ccomma &t>0.
\end{cases}
\end{eqnarray*}
Then $(\mu_0, \nu_0)$ is $IF-$norm on $N$. We note that we will use this $IF-$norm in the examples of the paper.
\end{example}
Throughout the paper $(N,\mu,\nu)$ will denote an $IFNS$.
\begin{definition}\cite{lael}
A sequence $(u_k)$ in $(N,\mu,\nu)$ is said to be convergent to $a\in N$ and denoted by $u_k\to a$ if for every $\varepsilon>0$ and $t>0$ there exists $k_0\in \mathbb{N}$ such that $\mu(u_k-a,t)>1-\varepsilon$ and $\nu(u_k-a,t)<\varepsilon$ for all $k\geq k_0$.
\end{definition}
\begin{definition}\cite{karakus}
Let sequence $(u_k)$ be in $(N,\mu,\nu)$. We say that $(u_k)$ is statistically convergent to $a\in N$ with respect to fuzzy norm $(\mu,\nu)$ provided that, for every $\varepsilon>0$ and $t>0$,
\begin{equation*}
    \lim_{n\to\infty}\frac{1}{n}\left|\left\{k\leq n : \mu(u_k-a,t)\leq 1-\varepsilon \quad \textrm{or} \quad \nu(u_k-u_n,t)\geq\varepsilon\right\}\right|=0.
\end{equation*}
In this case we write $st_{\mu,\nu}-\lim u=a$.
\end{definition}
\begin{theorem}\label{karakusthm}\cite{karakus}
Let sequence $(u_k)$ be in $(N,\mu,\nu)$ and $a\in N$. Then, $st_{\mu,\nu}-\lim u=a$ if and only if $st-\lim \mu(u_k-a,t)=1$ and $st-\lim \nu(u_k-a,t)=0$ for each $t$.
\end{theorem}
\begin{theorem}\label{convergencetost}\cite{karakus}
Let sequence $(u_k)$ be in $(N,\mu,\nu)$ and $a\in N$. If $(u_k)$ is convergent to $a$, then $(u_k)$ is statistically convergent to $a$.
\end{theorem}
\begin{definition}\label{q-bounded}\cite{qbounded}
A sequence $(u_k)$ in $(N,\mu,\nu)$ is called q-bounded if $\lim_{t\rightarrow\infty}\inf_{k\in \mathbb{N}}\mu(u_k,t)=1$ and\linebreak $\lim_{t\rightarrow\infty}\sup_{k\in \mathbb{N}}\nu(u_k,t)=0$.
\end{definition}
\begin{definition}\cite{yavuzcesaro}\label{slowdfn}
A sequence $(u_k)$ is slowly oscillating if and only if for all $t>0$ and for all $\varepsilon\in (0,1)$ there exist
$\lambda>1$ and $m_0\in \mathbb{N}$, depending on $t$ and $\varepsilon$, such that
\begin{equation*}
    \mu(u_k-u_n,t)>1-\varepsilon \quad \textrm{and} \quad \nu(u_k-u_n,t)<\varepsilon
\end{equation*}
whenever $m_0\leq n< k\leq\floor{\lambda n}$.
\end{definition}
\begin{definition}\cite{yavuzlog}\label{logslowdfn}
A sequence $(u_k)$ is slowly oscillating with respect to logarithmic summability if and only if for all $t>0$ and for all $\varepsilon\in (0,1)$ there exist $\lambda>1$ and $m_0\in \mathbb{N}$, depending on $t$ and $\varepsilon$, such that
\begin{equation*}
    \mu(u_k-u_n,t)>1-\varepsilon \quad \textrm{and} \quad \nu(u_k-u_n,t)<\varepsilon
\end{equation*}
whenever $m_0\leq n< k\leq\floor{n^{\lambda }}$.
\end{definition}
\begin{theorem}\label{cboundedtoslowly}\cite{yavuzcesaro}
Let sequence $(u_k)$ be in $(N,\mu,\nu)$. If $\{k(u_{k}-u_{k-1})\}$ is q-bounded, then $(u_k)$ is slowly oscillating.
\end{theorem}
\begin{theorem}\label{lboundedtoslowly}\cite{yavuzlog}
Let sequence $(u_k)$ be in $(N,\mu,\nu)$. If $\{k\ln k(u_{k}-u_{k-1})\}$ is q-bounded, then $(u_k)$ is slowly oscillating with respect to logarithmic summability.
\end{theorem}
\begin{theorem}\cite{yavuzcesaro}\label{cesarotoordinary}
Let sequence $(u_k)$ be in $(N,\mu,\nu)$. If $(u_k)$ is Ces\`{a}ro summable to some $a\in N$ and slowly oscillating, then $(u_k)$ converges to $a$.
\end{theorem}
\begin{theorem}\cite{yavuzlog}\label{logtoordinary}
Let sequence $(u_k)$ be in $(N,\mu,\nu)$. If $(u_k)$ is logarithmic summable to some $a\in N$ and slowly oscillating with respect to logarithmic summability, then $(u_k)$ converges to $a$.
\end{theorem}
\section{Tauberian theorems for statistical Ces\`{a}ro summability in $IFNS$}
Now, we define statistical Ces\`{a}ro summability method in $IFNS$. For some other studies concerning the concepts of Ces\`{a}ro summability and statistical convergence see \cite{moriczcesaro,nuray1,nuray2,nuray3,yavuzcesaro}.
\begin{definition}
\noindent Let $(u_k)$ be a sequence in $(N,\mu,\nu)$. Ces\`{a}ro means $\sigma_k$ of $(u_k)$ is defined by
\begin{eqnarray*}
\sigma_k=\frac{1}{k}\sum_{j=1}^k u_j.
\end{eqnarray*}
We say that $(u_k)$ is statistically Ces\`{a}ro summable to $a\in N$ if $st_{\mu,\nu}-\lim \sigma=a$.
\end{definition}
In view of Theorem \ref{convergencetost} and \cite[Theorem 3.2]{yavuzcesaro}, convergence implies statistical Ces\`{a}ro summability in $IFNS$. But converse statement is not true in general by the next example.
\begin{example}\label{ex1}
Let
\begin{eqnarray*}
u_k=
\begin{cases}
1+(-1)^k+k^2, \quad &k=n^2\\
1+(-1)^k-(k-1)^2,\quad &k=n^2+1\\
1+(-1)^k,\quad& otherwise
\end{cases}
\end{eqnarray*}
be in $IF-$normed space $(\mathbb{R},\mu_0,\nu_0)$. Sequence $(u_k)$ is neither convergent nor statistically convergent in $(\mathbb{R},\mu_0,\nu_0)$. Besides, it is not Ces\`{a}ro summable.

Let us apply statistical Ces\`{a}ro summability to achieve a limit. Ces\`{a}ro means $(\sigma_k)$ of sequence $(u_k)$ is
\begin{eqnarray*}
\sigma_k=
\begin{cases}
1+\frac{1}{k}\sum_{j=1}^k(-1)^j+k, \quad &k=n^2\\[3mm]
1+\frac{1}{k}\sum_{j=1}^k(-1)^j,\quad& otherwise.
\end{cases}
\end{eqnarray*}
Sequence $(\sigma_k)$ is statistically convergent to $1$ since for each $t>0$ we have $st-\lim\mu_0\left(\sigma_k-1, t\right)=1$ and $st-\lim\nu_0\left(\sigma_k-1, t\right)=0$ where
\begin{eqnarray*}
\mu_0\left(\sigma_k-1, t\right)=
\begin{cases}
\frac{t}{t+\left|\frac{1}{k}\sum_{j=1}^k(-1)^j+k\right|}\ccomma \ &k=n^2\\
\frac{t}{t+\left|\frac{1}{k}\sum_{j=1}^k(-1)^j\right|}\ccomma\ & otherwise
\end{cases}
\ \ and \ \
\nu_0\left(\sigma_k-1, t\right)=
\begin{cases}
\frac{\left|\frac{1}{k}\sum_{j=1}^k(-1)^j+k\right|}{t+\left|\frac{1}{k}\sum_{j=1}^k(-1)^j+k\right|}\ccomma \ &k=n^2\\[5mm]
\frac{\left|\frac{1}{k}\sum_{j=1}^k(-1)^j\right|}{t+\left|\frac{1}{k}\sum_{j=1}^k(-1)^j\right|}\ccomma\ & otherwise.
\end{cases}
\end{eqnarray*}
Hence, sequence $(u_k)$ is statistically Ces\`{a}ro summable to $1$ in $(\mathbb{R},\mu_0,\nu_0)$.
\end{example}
In this section, we will give some Tauberian conditions for statistical Ces\`{a}ro summability to imply convergence in $IFNS$. To this end, firstly we now show that statistical convergence of slowly oscillating sequences yields convergence in $IFNS$.
\begin{theorem}\label{sttoordinary}
Let sequence $(u_k)$ be in $(N,\mu,\nu)$. If $(u_k)$ is statistically convergent to some $a\in N$ and slowly oscillating, then $(u_k)$ is convergent to $a$.
\end{theorem}
\begin{proof}
Let $st_{\mu,\nu}-\lim u=a$ and sequence $(u_k)$ be slowly oscillating. Then by Theorem \ref{karakusthm}, for every $t>0$ we have $st-\lim \mu\left(u_k-a,t\right)=1$ and $st-\lim \nu\left(u_k-a,t\right)=0$. Our aim to show that $\lim_{k\to\infty} \mu\left(u_k-a,t\right)=1$ and $\lim_{k\to\infty} \nu\left(u_k-a,t\right)=0$.

Fix $t>0$. Since $st-\lim \mu\left(u_k-a,\frac{t}{2}\right)=1$, from the proof of \cite[Lemma 6]{moriczcesaro} there is a subsequence of integers $1\leq l_1<l_2<\cdots$ such that for any $\lambda>1$ inequality $l_m<l_{m+1}<\lambda l_m$ holds for large enough $m$; and
\begin{eqnarray*}
\lim_{m\to\infty}\mu\left(u_{l_m}-a,\frac{t}{2}\right)=1.
\end{eqnarray*}
So, for given $\varepsilon>0$ there exists $m_1$ such that
\begin{eqnarray}\label{sttoordinary-1}
\mu\left(u_{l_m}-a,\frac{t}{2}\right)>1-\varepsilon \qquad whenever\qquad m>m_1.
\end{eqnarray}
Besides, since $(u_k)$ is slowly oscillating there exist $m_2$ and $\lambda >1$ such that
\begin{eqnarray}\label{sttoordinary-2}
\mu\left(u_k-u_{l_m},\frac{t}{2}\right)>1-\varepsilon\qquad whenever\qquad m_2\leq l_m< k\leq \lambda l_m.
\end{eqnarray}
It follows from (\ref{sttoordinary-1})--(\ref{sttoordinary-2}) that
\begin{eqnarray*}
\mu\left(u_k-a,t\right)\geq\min\left\{\mu\left(u_k-u_{l_m},\frac{t}{2}\right),\mu\left(u_{l_m}-a,\frac{t}{2}\right)\right\}>1-\varepsilon
\end{eqnarray*}
holds for $l_m<k\leq l_{m+1}$ where $m>\max\{m_1, m_2\}$. By considering all $m$'s, we get
\begin{eqnarray*}
\mu\left(u_k-a,t\right)>1-\varepsilon \qquad for \qquad k>l_{m_3},
\end{eqnarray*}
where  $m_3=\max\{m_1,m_2\}$. Hence $\lim_{k\to\infty} \mu\left(u_k-a,t\right)=1$ is proved. The proof of $\lim_{k\to\infty} \nu\left(u_k-a,t\right)=0$ can be done similarly.
\end{proof}
We need next lemma to prove main theorem of this section.
\begin{lemma}\label{lemma}
Let $(u_k)$ be slowly oscillating sequence in $(N,\mu,\nu)$. Let $t$ be an arbitrary but fixed positive number. Then, for each $\varepsilon>0$ followings hold:
{\small\begin{eqnarray}\label{lemma1}
\mu\left(u_n-u_k,\left\{1+\frac{\ln\left(2n/k\right)}{\ln\lambda}\right\}t\right)>1-\varepsilon,\quad \nu\left(u_n-u_k,\left\{1+\frac{\ln\left(2n/k\right)}{\ln\lambda}\right\}t\right)<\varepsilon
\end{eqnarray}}
and
{\small\begin{eqnarray}\label{lemma2}
\mu\left(\frac{1}{n}\sum_{k=m_0}^{\floor{n/\lambda}}(u_n-u_k),\left\{1+\frac{\ln2+1}{\ln\lambda}\right\}t\right)>1-\varepsilon, \quad \nu\left(\frac{1}{n}\sum_{k=m_0}^{\floor{n/\lambda}}(u_n-u_k),\left\{1+\frac{\ln2+1}{\ln\lambda}\right\}t\right)<\varepsilon
\end{eqnarray}}
where $m_0\leq k\leq n/\lambda$, and $m_0=m_0(t,\varepsilon)$ and $\lambda=\lambda(t,\varepsilon)$ are from definition of slow oscillation.
\end{lemma}
\begin{proof}
Let $(u_k)$ be slowly oscillating sequence in $(N,\mu,\nu)$, and $m_0$ and $\lambda>1$ be from Definition\ref{slowdfn}. Let t be an arbitrary but fixed positive number. Fix $m_0\leq k\leq n/\lambda$. Consider the sequence(see \cite[Proof of Lemma 8]{moriczcesaro})
\begin{eqnarray*}
n_0:=n, \qquad n_p:=1+ \floor*{\frac{n_{p-1}}{\lambda}}, \qquad p=1,2,\ldots, q+1,
\end{eqnarray*}
where $q$ is determined by the condition
\begin{eqnarray*}
n_{q+1}\leq k<n_q.
\end{eqnarray*}
Then, we get
\begin{eqnarray*}
1-\varepsilon&<&\min\left\{\mu(u_n-u_{n_1},t),\mu(u_{n_1}-u_{n_2},t),\cdots,\mu(u_{n_q}-u_{k},t)\right\}
\\&\leq&
\mu\left(u_n-u_k,(q+1)t\right)
\\&\leq&
\mu\left(u_n-u_k,\left\{1+\frac{\ln\left(2n/k\right)}{\ln\lambda}\right\}t\right)
\end{eqnarray*}
and
\begin{eqnarray*}
\varepsilon&>&\max\left\{\nu(u_n-u_{n_1},t),\nu(u_{n_1}-u_{n_2},t),\cdots,\nu(u_{n_q}-u_{k},t)\right\}\geq \nu\left(u_n-u_k,\left\{1+\frac{\ln\left(2n/k\right)}{\ln\lambda}\right\}t\right)
\end{eqnarray*}
by virtue of the inequality $q\leq\frac{1}{\ln \lambda}\ln\left(\frac{2 n}{ k}\right)$ which was calculated in \cite[Proof of Lemma 8]{moriczcesaro}. This proves \eqref{lemma1}.

On the other hand by using \eqref{lemma1} we have:
\begin{eqnarray*}
1-\varepsilon&<&\min\limits_{m_0\leq k \leq \floor{n/\lambda}}\mu\left(u_n-u_k,\left\{1+\frac{\ln\left(2n/k\right)}{\ln\lambda}\right\}t\right)
\\&\leq&
\mu\left(\sum_{k=m_0}^{\floor{n/\lambda}}(u_n-u_k),\left\{\sum_{k=m_0}^{\floor{n/\lambda}}\left(1+\frac{\ln\left(2n/k\right)}{\ln\lambda}\right)\right\}t\right)
\\&\leq&
\mu\left(\sum_{k=m_0}^{\floor{n/\lambda}}(u_n-u_k),\left\{\sum_{k=1}^n\left(1+\frac{\ln\left(2n/k\right)}{\ln\lambda}\right)\right\}t\right)
\\&\leq&
\mu\left(\sum_{k=m_0}^{\floor{n/\lambda}}(u_n-u_k),\left\{1+\frac{\ln2+1}{\ln\lambda}\right\}nt\right)
\\&=&
\mu\left(\frac{1}{n}\sum_{k=m_0}^{\floor{n/\lambda}}(u_n-u_k),\left\{1+\frac{\ln2+1}{\ln\lambda}\right\}t\right)
\end{eqnarray*}
and
\begin{eqnarray*}
\varepsilon>\max\limits_{m_0\leq k \leq \floor{n/\lambda}}\nu\left(u_n-u_k,\left\{1+\frac{\ln\left(2n/k\right)}{\ln\lambda}\right\}t\right)
\geq \nu\left(\frac{1}{n}\sum_{k=m_0}^{\floor{n/\lambda}}(u_n-u_k),\left\{1+\frac{\ln2+1}{\ln\lambda}\right\}t\right)
\end{eqnarray*}
by virtue of $\sum_{k=2}^n\ln k>\int_1^n\ln(x)dx$, which proves \eqref{lemma2}
\end{proof}
\begin{theorem}\label{coscillatingtooscillating}
Let sequence $(u_k)$ be in $(N,\mu,\nu)$. If $(u_k)$ is slowly oscillating, then sequence $(\sigma_k)$ of Ces\`{a}ro means is also slowly oscillating.
\end{theorem}
\begin{proof}
Let sequence $(u_k)$ be slowly oscillating and $\sigma_k=\frac{1}{k}\sum_{j=1}^ku_j$. Fix $t>0$. For given $\varepsilon>0$ there exists $m_0,m_1\in\mathbb{N}$ and $1<\lambda<2$ such that
\begin{itemize}
\item  $\mu(u_k-u_n,t/16)>1-\varepsilon$ and $\nu(u_k-u_n,t/16)<\varepsilon$ whenever $m_0\leq n< k\leq\floor{\lambda n}$.
\item
\begin{equation*}
\mu\left(\dfrac{k-n}{kn}\sum\limits_{j=1}^{m_0-1}(u_n-u_j),\dfrac{t}{4}\right)>1-\varepsilon \quad \textrm{and} \quad \nu\left(\dfrac{k-n}{kn}\sum\limits_{j=1}^{m_0-1}(u_n-u_j),\dfrac{t}{4}\right)<\varepsilon
\end{equation*}
whenever $m_1\leq n< k\leq\floor{\lambda n}$, by virtue of inequalities in \eqref{lemma1}.
\end{itemize}
Then, for $\max\{m_0,m_1\}\leq n< k\leq\floor{\lambda n}$, we get
\begin{eqnarray*}
\mu(\sigma_k-\sigma_n,t)&=&\mu\left(\frac{k-n}{kn}\sum_{j=1}^{n}(u_n-u_j)+\frac{1}{k}\sum_{j=n+1}^{k}(u_j-u_n),t\right)
\\&\geq&
\min\left\{\mu\left(\frac{k-n}{kn}\sum_{j=1}^{m_0-1}(u_n-u_j),\frac{t}{4}\right),\mu\left(\frac{k-n}{kn}\sum_{j=m_0}^{\floor{n/\lambda}}(u_n-u_j),\frac{t}{4}\right),\right.
\\&&\hspace{2cm}
\left.\mu\left(\frac{k-n}{kn}\sum_{j=\floor{n/\lambda}+1}^{n}\!\!\!\!\!(u_n-u_j),\frac{t}{4}\right),\mu\left(\frac{1}{k}\sum_{j=n+1}^{k}(u_j-u_n),\frac{t}{4}\right)\right\}
\\&\geq&
\min\left\{\mu\left(\frac{k-n}{kn}\sum_{j=1}^{m_0-1}(u_n-u_j),\frac{t}{4}\right),\mu\left(\frac{k-n}{kn}\sum_{j=m_0}^{\floor{n/\lambda}}(u_n-u_j),\frac{t}{4}\right),\right.
\\&&\hspace{2cm}
\left.\min_{\floor{n/\lambda}\leq j\leq n}\mu\left(u_n-u_j,\frac{t}{4}\right),\min_{n+1\leq j\leq k}\mu\left(u_j-u_n,\frac{t}{4}\right)\right\}
\\&>&
\min\left\{\mu\left(\frac{k-n}{kn}\sum_{j=m_0}^{\floor{n/\lambda}}(u_n-u_j),\frac{t}{4}\right), 1-\varepsilon\right\}
\\&\geq&
\min\left\{\mu\left(\frac{1}{n}\sum_{j=m_0}^{\floor{n/\lambda}}(u_n-u_j),\frac{t}{4(\lambda-1)}\right), 1-\varepsilon\right\}
\\&\geq&
\min\left\{\mu\left(\frac{1}{n}\sum_{j=m_0}^{\floor{n/\lambda}}(u_n-u_j),\left\{1+\frac{\ln2+1}{\ln\lambda}\right\}\frac{t}{16}\right), 1-\varepsilon\right\}
\\&\geq&
1-\varepsilon
\end{eqnarray*}
by virtue of the facts that $\frac{k-n}{k}<\lambda-1$ and $\frac{\lambda-1}{\ln\lambda}(\ln\lambda+\ln2+1)<4$, and of Lemma \ref{lemma}. On the other hand, $\nu(\sigma_k-\sigma_n,t)<\varepsilon$ can be shown similarly. Hence, the proof is completed.
\end{proof}
Now we give the main theorem of this section.
\begin{theorem}\label{cesaromain}
Let sequence $(u_k)$ be in $(N,\mu,\nu)$. If $(u_k)$ is statistically Ces\`{a}ro summable to some $a\in N$ and slowly oscillating, then $(u_k)$ is convergent to $a$.
\end{theorem}
\begin{proof}
Let $(u_k)$ be statistically Ces\`{a}ro summable to some $a\in N$ and slowly oscillating. Then, by Theorem \ref{coscillatingtooscillating} sequence $(\sigma_k)$ of Ces\`{a}ro means is also slowly oscillating. From Theorem \ref{sttoordinary}, $(\sigma_k)\to a$. This means that $(u_k)$ is Ces\`{a}ro summable to $a$. Since $(u_k)$ is slowly oscillating, from Theorem \ref{cesarotoordinary} $(u_k)$ is convergent to $a$.
\end{proof}
In view of Theorem \ref{cboundedtoslowly} and Theorem \ref{cesaromain} we get following theorem.
\begin{theorem}\label{cesaromainbnd}
Let sequence $(u_k)$ be in $(N,\mu,\nu)$. If $(u_k)$ is  statistically Ces\`{a}ro summable to $a\in N$ and  $\{k(u_{k}-u_{k-1})\}$ is q-bounded, then  $(u_k)$ converges to $a$.
\end{theorem}
In some cases of sequences as in Example \ref{ex2}, first order Ces\`{a}ro means fails to converge both ordinarily and statistically in $IFNS$. To handle such sequences, we consider higher order Ces\`{a}ro means and define statistical H\"{o}lder summability method in $IFNS$. In the sequel, we give corresponding Tauberian theorems as corollaries.
\begin{definition}
Let $(u_k)$ be in $(N,\mu,\nu)$. $m$-th order H\"{o}lder means $H^m_k$ of $(u_k)$ is defined by
\begin{eqnarray*}
H^m_k=\frac{1}{k}\sum_{j=1}^{k}H^{m-1}_j,
\end{eqnarray*}
where $H_k^0=u_k$. Sequence $(u_k)$ is said to be $(H,m)$ summable to $a\in N$ if $\lim_{k\to\infty}H^{m}_k=a$ and
it is said to be statistically $(H,m)$ summable to $a$ if $st_{\mu,\nu}-\lim H^m=a$
\end{definition}
\begin{example}\label{ex2}
Let
\begin{eqnarray*}
u_k=
\begin{cases}
(-1)^kk^2+k^4, \quad &k=n^2\\
(-1)^kk^2-(k-1)^4-(k-1)^3k-(k-1)^2k^2,\quad &k=n^2+1\\
(-1)^kk^2+3(k-2)^2(k-1)^2,\quad &k=n^2+2\\
(-1)^kk^2-(k-3)^2(k-2)(k-1),\quad &k=n^2+3\\
(-1)^kk^2,\quad& otherwise
\end{cases}
\end{eqnarray*}
be in $IF-$normed space $(\mathbb{R},\mu_0,\nu_0)$ where $n\geq2$. Sequence $(u_k)$ is neither convergent nor statistically convergent in $(\mathbb{R},\mu_0,\nu_0)$. Furthermore, it is neither Ces\`{a}ro summable nor statistically Ces\`{a}ro summable.

Let us apply statistical $(H,3)$ summability to achieve a limit. H\"{o}lder means $(H^1_k)$, $(H^2_k)$ and $(H^3_k)$ of sequence $(u_k)$ are
\begin{eqnarray*}
H^1_k=
\begin{cases}
\sigma^{\mathsmaller{(1)}}_k+k^3, \quad &k=n^2\\
\sigma^{\mathsmaller{(1)}}_k-(k-1)^2(2k-1),\quad &k=n^2+1\\
\sigma^{\mathsmaller{(1)}}_k+(k-2)^2(k-1),\quad &k=n^2+2\\
\sigma^{\mathsmaller{(1)}}_k,\quad& otherwise
\end{cases}
\end{eqnarray*}
\begin{eqnarray*}
H^2_k=
\begin{cases}
\sigma^{\mathsmaller{(2)}}_k+k^2, \quad &k=n^2\\
\sigma^{\mathsmaller{(2)}}_k-(k-1)^2,\quad &k=n^2+1\\
\sigma^{\mathsmaller{(2)}}_k,\quad& otherwise
\end{cases}
\end{eqnarray*}
\begin{eqnarray*}
H^3_k=
\begin{cases}
\sigma^{\mathsmaller{(3)}}_k+k, \quad &k=n^2\\
\sigma^{\mathsmaller{(3)}}_k,\quad& otherwise
\end{cases}
\end{eqnarray*}
where sequence $\left\{\sigma^{\mathsmaller{(m)}}_k\right\}$ denotes $m-$fold Ces\`{a}ro means of sequence $\left\{(-1)^kk^2\right\}$ and $n\geq2$. Sequence $(H^3_k)$ is statistically convergent to 0 since for each $t>0$ we have $st-\lim\mu_0\left(H^3_k, t\right)=1$ and $st-\lim\nu_0\left(H^3_k, t\right)=0$ where
\begin{eqnarray*}
\mu_0\left(H^3_k, t\right)=
\begin{cases}
\frac{t}{t+\left|\sigma^{\mathsmaller{(3)}}_k+k\right|}\ccomma \quad &k=n^2\\[5mm]
\frac{t}{t+\left|\sigma^{\mathsmaller{(3)}}_k\right|}\ccomma\quad& otherwise
\end{cases}
\qquad and \qquad
\nu_0\left(H^3_k, t\right)=
\begin{cases}
\frac{\left|\sigma^{\mathsmaller{(3)}}_k+k\right|}{t+\left|\sigma^{\mathsmaller{(3)}}_k+k\right|}\ccomma \quad &k=n^2\\[5mm]
\frac{\left|\sigma^{\mathsmaller{(3)}}_k\right|}{t+\left|\sigma^{\mathsmaller{(3)}}_k\right|}\ccomma\quad& otherwise.
\end{cases}
\end{eqnarray*}
in view of the fact that $\lim_{k\to\infty}\sigma^{\mathsmaller{(3)}}_k=0$. Hence, sequence $(u_k)$ is statistically $(H,3)$ summable to 0 in $(\mathbb{R},\mu_0,\nu_0)$.
\end{example}
In view of Theorem \ref{cesarotoordinary}, Theorem \ref{coscillatingtooscillating} and Theorem \ref{cesaromain}  we get following Tauberian theorem.
\begin{theorem}
Let sequence $(u_k)$ be in $(N,\mu,\nu)$. If $(u_k)$ is statistically $(H,m)$ summable to some $a\in N$ and slowly oscillating, then $(u_k)$ is convergent to $a$.
\end{theorem}
Also, in view of theorem above and Theorem \ref{cboundedtoslowly} we get following theorem.
\begin{theorem}
Let sequence $(u_k)$ be in $(N,\mu,\nu)$. If $(u_k)$ is statistically $(H,m)$ summable to some $a\in N$ and $\{k(u_{k}-u_{k-1})\}$ is q-bounded, then $(u_k)$ is convergent to $a$.
\end{theorem}
\section{Tauberian theorems for statistical logarithmic summability in $IFNS$}
We define statistical logarithmic summability method in $IFNS$ as the following.
\begin{definition}
\noindent Let sequence $(u_k)$ be in $(N,\mu,\nu)$. Logarithmic mean $\tau_k$ of $(u_k)$ is defined by
\begin{eqnarray*}
\tau_k=\frac{1}{\ell_k}\sum_{j=1}^k \frac{u_j}{j} \quad where\quad \ell_k=\sum_{j=1}^k\frac{1}{j}\ \cdot
\end{eqnarray*}
$(u_k)$ is said to be statistically logarithmic summable to $a\in N$ if $st_{\mu,\nu}-\lim \tau=a$.
\end{definition}
In view of Theorem \ref{convergencetost} and \cite[Theorem 2.2]{yavuzlog}, convergence implies statistical logarithmic summability in $IFNS$. But converse statement is not true in general by the next example.
\begin{example}\label{ex3}
Consider the vector space $C[0,1]$ equipped with the norm $\Vert f\Vert=\max_{x\in[0,1]}|f(x)|$. Let
\begin{eqnarray*}
f_k(x)=
\begin{cases}
(-x)^kk+k\left(\ell_k\right)^2, \quad &k=n^2\\
(-x)^kk-k\left(\ell_{k-1}\right)^2,\quad &k=n^2+1\\
(-x)^kk,\quad& otherwise
\end{cases}
\end{eqnarray*}
be in $IF-$normed space $(C[0,1],\mu_0,\nu_0)$ where $\ell_k=\sum_{j=1}^k 1/j$. Sequence $(f_k)$ is neither convergent nor statistically convergent in $(C[0,1],\mu_0,\nu_0)$. Besides, $(f_k)$ is neither Ces\`{a}ro summable nor logarithmic summable. Furthermore, $(f_k)$ is not statistically Ces\`{a}ro summable which we have defined in previous section.

Let us apply statistical logarithmic summability to achieve a limit. Logarithmic means $(\tau_k)$ of sequence $(f_k)$ is
\begin{eqnarray*}
\tau_k(x)=
\begin{cases}
\frac{1}{\ell_k}\sum_{j=1}^k(-x)^j+\ell_k, \quad &k=n^2\\[3mm]
\frac{1}{\ell_k}\sum_{j=1}^k(-x)^j,\quad& otherwise.
\end{cases}
\end{eqnarray*}
Hence,  $(f_k)$ is statistically logarithmic summable to 0 since for each $t>0$ we have $st-\lim\mu_0\left(\tau_k, t\right)=1$ and $st-\lim\nu_0\left(\tau_k, t\right)=0$ where
\begin{eqnarray*}
\mu_0\left(\tau_k, t\right)=
\begin{cases}
\frac{t}{t+\left\Vert\frac{1}{\ell_k}\sum\limits_{j=1}^k(-x)^j+\ell_k\right\Vert}\ccomma \quad &k=n^2\\[5mm]
\frac{t}{t+\left\Vert\frac{1}{\ell_k}\sum\limits_{j=1}^k(-x)^j\right\Vert}\ccomma\quad& otherwise
\end{cases}
\quad and \quad
\nu_0\left(\tau_k, t\right)=
\begin{cases}
\frac{\left\Vert\frac{1}{\ell_k}\sum\limits_{j=1}^k(-x)^j+\ell_k\right\Vert}{t+\left\Vert\frac{1}{\ell_k}\sum\limits_{j=1}^k(-x)^j+\ell_k\right\Vert}\ccomma \quad &k=n^2\\[8mm]
\frac{\left\Vert\frac{1}{\ell_k}\sum\limits_{j=1}^k(-x)^j\right\Vert}{t+\left\Vert\frac{1}{\ell_k}\sum\limits_{j=1}^k(-x)^j\right\Vert}\ccomma\quad& otherwise.
\end{cases}
\end{eqnarray*}
\end{example}
In this section, we will give some Tauberian conditions for statistical logarithmic summability to imply convergence in $IFNS$. To this end, firstly we now show that statistical convergence of slowly oscillating sequences with respect to logarithmic summability implies convergence in $IFNS$.
\begin{theorem}\label{logsttoordinary}
Let sequence $(u_k)$ be in $(N,\mu,\nu)$. If $(u_k)$ is statistically convergent to some $a\in N$ and slowly oscillating with respect to logarithmic summability, then $(u_k)$ is convergent to $a$.
\end{theorem}
\begin{proof}
Let $st_{\mu,\nu}-\lim u=a$ and sequence $(u_k)$ be slowly oscillating with respect to logarithmic summability. Then by Theorem \ref{karakusthm}, for every $t>0$ we have $st-\lim \mu\left(u_k-a,t\right)=1$ and $st-\lim \nu\left(u_k-a,t\right)=0$. Our aim to show that $\lim_{k\to\infty} \mu\left(u_k-a,t\right)=1$ and $\lim_{k\to\infty} \nu\left(u_k-a,t\right)=0$.

Fix $t>0$. Since $st-\lim \mu\left(u_k-a,\frac{t}{2}\right)=1$, from the proof of \cite[Lemma 10]{moriczlog} there is a subsequence of integers $1\leq j_1<j_2<\cdots$ such that for any $\lambda>1$ inequality $j_m<j_{m+1}<j^{\lambda}_m$ holds for large enough $m$; and
\begin{eqnarray*}
\lim_{m\to\infty}\mu\left(u_{j_m}-a,\frac{t}{2}\right)=1.
\end{eqnarray*}
So, for $\varepsilon>0$ there exists $m_1$ so that
\begin{eqnarray}\label{logsttoordinary-1}
\mu\left(u_{j_m}-a,\frac{t}{2}\right)>1-\varepsilon \qquad whenever\qquad m>m_1.
\end{eqnarray}
Besides, since $(u_k)$ is slowly oscillating with respect to logarithmic summability there exist $m_2$ and $\lambda >1$ such that
\begin{eqnarray}\label{logsttoordinary-2}
\mu\left(u_k-u_{j_m},\frac{t}{2}\right)>1-\varepsilon\qquad whenever\qquad m_2\leq j_m< k\leq j^{\lambda}_m.
\end{eqnarray}
It follows from (\ref{logsttoordinary-1})--(\ref{logsttoordinary-2}) that
\begin{eqnarray*}
\mu\left(u_k-a,t\right)\geq\min\left\{\mu\left(u_k-u_{j_m},\frac{t}{2}\right),\mu\left(u_{j_m}-a,\frac{t}{2}\right)\right\}>1-\varepsilon
\end{eqnarray*}
holds for $j_m<k\leq j_{m+1}$ where $m>\max\{m_1, m_2\}$. By considering all $m$'s, we get
\begin{eqnarray*}
\mu\left(u_k-a,t\right)>1-\varepsilon \qquad for \qquad k>j_{m_3},
\end{eqnarray*}
where  $m_3=\max\{m_1,m_2\}$. Hence $\lim_{k\to\infty} \mu\left(u_k-a,t\right)=1$ is proved. The proof of $\lim_{k\to\infty} \nu\left(u_k-a,t\right)=0$ can be done similarly.
\end{proof}
We need the next lemma to prove main theorem of this section.
\begin{lemma}\label{loglemma}
Let sequence $(u_k)$ be slowly oscillating with respect to logarithmic summability in $(N,\mu,\nu)$. Let $t$ be an arbitrary but fixed positive number. Then, for each $\varepsilon>0$ followings hold:
{\small\begin{eqnarray}\label{loglemma1}
\mu\left(u_n-u_k,\left\{1+\frac{1}{\ln\lambda}\ln\left(\frac{2\ln n}{\ln k}\right)\right\}t\right)>1-\varepsilon,\quad \nu\left(u_n-u_k,\left\{1+\frac{1}{\ln\lambda}\ln\left(\frac{2\ln n}{\ln k}\right)\right\}t\right)<\varepsilon
\end{eqnarray}}
and
{\small\begin{eqnarray}\label{loglemma2}
\mu\left(\frac{1}{\ell_n}\sum_{k=m_0}^{\floor{n^{1/\lambda}}}\!\!\left(\frac{u_n-u_k}{k}\right),\left\{1+\frac{\ln2+2}{\ln\lambda}\right\}t\right)>1-\varepsilon, \quad \nu\left(\frac{1}{\ell_n}\sum_{k=m_0}^{\floor{n^{1/\lambda}}}\!\!\left(\frac{u_n-u_k}{k}\right),\left\{1+\frac{\ln2+2}{\ln\lambda}\right\}t\right)<\varepsilon
\end{eqnarray}}%
where $1<m_0\leq k\leq n^{1/\lambda}$, and $m_0=m_0(t,\varepsilon)$ and $\lambda=\lambda(t,\varepsilon)$ are from definition of slow oscillation with respect to logarithmic summability.
\end{lemma}
\begin{proof}
Let sequence $(u_k)$ be slowly oscillating with respect to logarithmic summability in $(N,\mu,\nu)$, and $1<m_0$ and $\lambda>1$ be from Definition \ref{logslowdfn}. Let $t$ be an arbitrary but fixed positive number. Fix $m_0\leq k\leq n^{1/\lambda}$. Consider the sequence(see \cite[Proof of Lemma 5]{moriczlog})
\begin{eqnarray*}
n_0:=n, \qquad n_p:=1+ \floor*{n^{1/\lambda}_{p-1}}, \qquad p=1,2,\ldots, q+1,
\end{eqnarray*}
where $q$ is determined by the condition
\begin{eqnarray*}
n_{q+1}\leq k<n_q.
\end{eqnarray*}
Then, we get
\begin{eqnarray*}
1-\varepsilon&<&\min\left\{\mu(u_n-u_{n_1},t),\mu(u_{n_1}-u_{n_2},t),\cdots,\mu(u_{n_q}-u_{k},t)\right\}
\\&\leq&
\mu\left(u_n-u_k,(q+1)t\right)
\\&\leq&
\mu\left(u_n-u_k,\left\{1+\frac{1}{\ln\lambda}\ln\left(\frac{2\ln n}{\ln k}\right)\right\}t\right)
\end{eqnarray*}
and
\begin{eqnarray*}
\varepsilon&>&\max\left\{\nu(u_n-u_{n_1},t),\nu(u_{n_1}-u_{n_2},t),\cdots,\nu(u_{n_q}-u_{k},t)\right\}\geq \nu\left(u_n-u_k,\left\{1+\frac{1}{\ln\lambda}\ln\left(\frac{2\ln n}{\ln k}\right)\right\}t\right)
\end{eqnarray*}
by virtue of the inequality $q\leq\frac{1}{\ln\lambda}\ln\left(\frac{2\ln n}{\ln k}\right)$ which was calculated in \cite[Proof of Lemma 5]{moriczlog}. This proves \eqref{loglemma1}.

On the other hand by using \eqref{loglemma1} we have:
\begin{eqnarray*}
1-\varepsilon&<&\min\limits_{m_0\leq k \leq \floor{n^{1/\lambda}}}\mu\left(u_n-u_k,\left\{1+\frac{1}{\ln\lambda}\ln\left(\frac{2\ln n}{\ln k}\right)\right\}t\right)
\\&\leq&
\mu\left(\sum_{k=m_0}^{\floor{n^{1/\lambda}}}\!\!\left(\frac{u_n-u_k}{k}\right),\left\{\sum_{k=m_0}^{\floor{n^{1/\lambda}}}\frac{1}{k}\left(1+\frac{1}{\ln\lambda}\ln\left(\frac{2\ln n}{\ln k}\right)\right)\right\}t\right)
\\&\leq&
\mu\left(\sum_{k=m_0}^{\floor{n^{1/\lambda}}}\!\!\left(\frac{u_n-u_k}{k}\right),\left\{\sum_{k=2}^n\frac{1}{k}\left(1+\frac{1}{\ln\lambda}\ln\left(\frac{2\ln n}{\ln k}\right)\right)\right\}t\right)
\\&\leq&
\mu\left(\sum_{k=m_0}^{\floor{n^{1/\lambda}}}\!\!\left(\frac{u_n-u_k}{k}\right),\left\{1+\frac{\ln2+2}{\ln\lambda}\right\}\ell_nt\right)
\\&=&
\mu\left(\frac{1}{\ell_n}\sum_{k=m_0}^{\floor{n^{1/\lambda}}}\!\!\left(\frac{u_n-u_k}{k}\right),\left\{1+\frac{\ln2+2}{\ln\lambda}\right\}t\right)
\end{eqnarray*}
and
\begin{eqnarray*}
\varepsilon>\max\limits_{m_0\leq k \leq \floor{n^{1/\lambda}}}\nu\left(u_n-u_k,\left\{1+\frac{1}{\ln\lambda}\ln\left(\frac{2\ln n}{\ln k}\right)\right\}t\right)
\geq \nu\left(\frac{1}{\ell_n}\sum_{k=m_0}^{\floor{n^{1/\lambda}}}\!\!\left(\frac{u_n-u_k}{k}\right),\left\{1+\frac{\ln2+2}{\ln\lambda}\right\}t\right)
\end{eqnarray*}
which proves \eqref{loglemma2}.
\end{proof}
\begin{theorem}\label{loscillatingtooscillating}
Let sequence $(u_k)$ be in $(N,\mu,\nu)$. If $(u_k)$ is slowly oscillating with respect to logarithmic summability, then sequence $(\tau_k)$ of logarithmic means is also slowly oscillating with respect to logarithmic summability.
\end{theorem}
\begin{proof}
Let sequence $(u_k)$ be slowly oscillating with respect to logarithmic summability and $\tau_k=\frac{1}{\ell_k}\sum_{j=1}^k\frac{u_j}{j}$. Fix $t>0$. For given $\varepsilon>0$ there exists $1<m_0,m_1\in\mathbb{N}$ and $1<\lambda<2$ such that
\begin{itemize}
\item  $\mu(u_k-u_n,t/20)>1-\varepsilon$ and $\nu(u_k-u_n,t/20)<\varepsilon$ whenever $m_0\leq n< k\leq\floor{n^{\lambda}}$.
\item
\begin{equation*}
\mu\left(\left(\dfrac{1}{\ell_n}-\dfrac{1}{\ell_k}\right)\sum\limits_{j=1}^{m_0-1}\!\!\left(\frac{u_n-u_j}{j}\right),\dfrac{t}{4}\right)>1-\varepsilon \quad \textrm{and} \quad \nu\left(\left(\dfrac{1}{\ell_n}-\dfrac{1}{\ell_k}\right)\sum\limits_{j=1}^{m_0-1}\!\!\left(\frac{u_n-u_j}{j}\right),\dfrac{t}{4}\right)<\varepsilon
\end{equation*}%
whenever $m_1\leq n< k\leq\floor{n^{\lambda}}$, by virtue of inequalities in \eqref{loglemma1}.
\end{itemize}
Then, for $\max\{m_0,m_1\}\leq n< k\leq\floor{n^{\lambda}}$, we get
\begin{eqnarray*}
\mu(\tau_k-\tau_n,t)&=&\mu\left(\left(\frac{1}{\ell_n}-\dfrac{1}{\ell_k}\right)\sum_{j=1}^{n}\!\!\left(\frac{u_n-u_j}{j}\right)+\frac{1}{\ell_k}\sum_{j=n+1}^{k}\!\!\left(\frac{u_j-u_n}{j}\right),t\right)
\\&\geq&
\min\left\{\mu\left(\left(\frac{1}{\ell_n}-\dfrac{1}{\ell_k}\right)\sum_{j=1}^{m_0-1}\!\!\left(\frac{u_n-u_j}{j}\right),\frac{t}{4}\right),\mu\left(\left(\frac{1}{\ell_n}-\dfrac{1}{\ell_k}\right)\sum_{j=m_0}^{\floor{n^{1/\lambda}}}\!\!\left(\frac{u_n-u_j}{j}\right),\frac{t}{4}\right),\right.
\\&&\hspace{2cm}
\left.\mu\left(\left(\frac{1}{\ell_n}-\dfrac{1}{\ell_k}\right)\sum_{j=\floor{n^{1/\lambda}}+1}^{n}\!\!\left(\frac{u_n-u_j}{j}\right),\frac{t}{4}\right),\mu\left(\dfrac{1}{\ell_k}\sum_{j=n+1}^{k}\!\!\left(\frac{u_n-u_j}{j}\right),\frac{t}{4}\right)\right\}
\\&\geq&
\min\left\{\mu\left(\left(\dfrac{1}{\ell_n}-\dfrac{1}{\ell_k}\right)\sum_{j=1}^{m_0-1}\!\!\left(\frac{u_n-u_j}{j}\right),\frac{t}{4}\right),\mu\left(\left(\dfrac{1}{\ell_n}-\dfrac{1}{\ell_k}\right)\sum_{j=m_0}^{\floor{n^{1/\lambda}}}\!\!\left(\frac{u_n-u_j}{j}\right),\frac{t}{4}\right),\right.
\\&&\hspace{2cm}
\left.\min_{\floor{n^{1/\lambda}}\leq j\leq n}\mu\left(u_n-u_j,\frac{t}{4}\right),\min_{n+1\leq j\leq k}\mu\left(u_j-u_n,\frac{t}{4}\right)\right\}
\\&>&
\min\left\{\mu\left(\left(\frac{1}{\ell_n}-\dfrac{1}{\ell_k}\right)\sum_{j=m_0}^{\floor{n^{1/\lambda}}}\!\!\left(\frac{u_n-u_j}{j}\right),\frac{t}{4}\right), 1-\varepsilon\right\}
\\&\geq&
\min\left\{\mu\left(\frac{1}{\ell_n}\sum_{j=m_0}^{\floor{n^{1/\lambda}}}\!\!\left(\frac{u_n-u_j}{j}\right),\frac{t}{4(\lambda-1)}\right), 1-\varepsilon\right\}
\\&\geq&
\min\left\{\mu\left(\frac{1}{\ell_n}\sum_{j=m_0}^{\floor{n^{1/\lambda}}}\!\!\left(\frac{u_n-u_j}{j}\right),\left\{1+\frac{\ln2+2}{\ln\lambda}\right\}\frac{t}{20}\right), 1-\varepsilon\right\}
\\&\geq&
1-\varepsilon
\end{eqnarray*}
by virtue of the facts that $\frac{\ell_k-\ell_n}{\ell_k}\leq\frac{\ln k-\ln n}{\ln k}<\lambda-1$ and $\frac{\lambda-1}{\ln\lambda}(\ln\lambda+\ln2+2)<5$, and of Lemma \ref{loglemma}. On the other hand, $\nu(\tau_k-\tau_n,t)<\varepsilon$ can be shown similarly. Hence, the proof is completed.
\end{proof}
Now we give the main theorem of this section.
\begin{theorem}\label{logmain}
Let sequence $(u_k)$ be in $(N,\mu,\nu)$. If $(u_k)$ is statistically logarithmic summable to some $a\in N$ and slowly oscillating with respect to logarithmic summability, then $(u_k)$ is convergent to $a$.
\end{theorem}
\begin{proof}
Let $(u_k)$ be statistically logarithmic summable to some $a\in N$ and slowly oscillating with respect to logarithmic summability. Then, by Theorem \ref{loscillatingtooscillating} sequence $(\tau_k)$ of logarithmic means is also slowly oscillating with respect to logarithmic summability. From Theorem \ref{logsttoordinary}, $(\tau_k)\to a$. This means that $(u_k)$ is logarithmic summable to $a$. Since $(u_k)$ is slowly oscillating with respect to logarithmic summability, from Theorem \ref{logtoordinary} $(u_k)$ is convergent to $a$.
\end{proof}
In view of Theorem \ref{lboundedtoslowly} and Theorem \ref{logmain} we get following theorem.
\begin{theorem}
Let sequence $(u_k)$ be in $(N,\mu,\nu)$. If $(u_k)$ is logarithmic summable to $a\in N$ and  $\{k\ln k(u_{k}-u_{k-1})\}$ is q-bounded, then  $(u_k)$ converges to $a$.
\end{theorem}
For the case of higher order logarithmic summability methods we can give following Tauberian theorems as corollaries.
\begin{definition}
Let $(u_k)$ be in $(N,\mu,\nu)$. $m$-th order logarithmic means $L^m_k$ of $(u_k)$ is defined by
\begin{eqnarray*}
L^m_k=\frac{1}{\ell_k}\sum_{j=1}^{k}\frac{L^{m-1}_j}{j},
\end{eqnarray*}
where $L_k^0=u_k$. We say that sequence $(u_k)$ is statistically $(L,m)$ summable to $a\in N$ if sequence $(L^{m}_k)$ is statistically convergent to $a$.
\end{definition}
\begin{example}\label{ex4}
Consider the vector space $C[0,1]$ equipped with the norm $\Vert f\Vert=\max_{x\in[0,1]}|f(x)|$. Let
\begin{eqnarray*}
f_k(x)=
\begin{cases}
(-x)^kk^2+k^2\left(\ell_k\right)^3, \quad &k=n^2\\
(-x)^kk^2-k\left(\ell_{k-1}\right)^2\left((k-1)\ell_{k-1}+k\ell_k\right),\quad &k=n^2+1\\
(-x)^kk^2+k(k-1)\ell_{k-1}\left(\ell_{k-2}\right)^2,\quad &k=n^2+2\\
(-x)^kk^2,\quad& otherwise
\end{cases}
\end{eqnarray*}
be in $IF-$normed space $(C[0,1],\mu_0,\nu_0)$ where $\ell_k=\sum_{j=1}^k 1/j$. Sequence $(f_k)$ is neither convergent nor statistically convergent in $(C[0,1],\mu_0,\nu_0)$. Besides, $(f_k)$ is neither statistical Ces\`{a}ro summable nor statistical logarithmic summable.

Let us apply statistical $(L,2)$ summability to achieve a limit. Logarithmic means $(L^1_k)$  and $(L^2_k)$ of sequence $(f_k)$ are
\begin{eqnarray*}
L^1_k(x)=
\begin{cases}
\tau^{\mathsmaller{(1)}}_k(x)+k(\ell_k)^2, \quad &k=n^2\\
\tau^{\mathsmaller{(1)}}_k(x)-k(\ell_{k-1})^2, \quad &k=n^2+1\\
\tau^{\mathsmaller{(1)}}_k(x),\quad& otherwise,
\end{cases}
\end{eqnarray*}
\begin{eqnarray*}
L^2_k(x)=
\begin{cases}
\tau^{\mathsmaller{(2)}}_k(x)+\ell_k, \quad &k=n^2\\
\tau^{\mathsmaller{(2)}}_k(x),\quad& otherwise
\end{cases}
\end{eqnarray*}
where sequence $\left\{\tau^{\mathsmaller{(m)}}_k\right\}$ denotes $m-$fold logarithmic means of sequence $\left\{(-x)^kk^2\right\}$. Hence,  $(f_k)$ is statistically $(L,2)$ summable to 0 since for each $t>0$ we have $st-\lim\mu_0\left(L^2_k, t\right)=1$ and $st-\lim\nu_0\left(L^2_k, t\right)=0$ where
\begin{eqnarray*}
\mu_0\left(L^2_k, t\right)=
\begin{cases}
\frac{t}{t+\left\Vert \tau^{\mathsmaller{(2)}}_k(x)+\ell_k\right\Vert}\ccomma \quad &k=n^2\\[3mm]
\frac{t}{t+\left\Vert \tau^{\mathsmaller{(2)}}_k(x)\right\Vert}\ccomma\quad& otherwise
\end{cases}
\quad and \quad
\nu_0\left(L^2_k, t\right)=
\begin{cases}
\frac{\left\Vert\tau^{\mathsmaller{(2)}}_k(x)\right\Vert}{t+\left\Vert\tau^{\mathsmaller{(2)}}_k(x)+\ell_k\right\Vert}\ccomma \quad &k=n^2\\[3mm]
\frac{\left\Vert\tau^{\mathsmaller{(2)}}_k(x)\right\Vert}{t+\left\Vert\tau^{\mathsmaller{(2)}}_k(x)\right\Vert}\ccomma\quad& otherwise.
\end{cases}
\end{eqnarray*}
\end{example}
In view of Theorem \ref{logtoordinary}, Theorem \ref{loscillatingtooscillating} and Theorem \ref{logmain}  we get following Tauberian theorem.
\begin{theorem}
Let sequence $(u_k)$ be in $(N,\mu,\nu)$. If $(u_k)$ is statistically $(L,m)$ summable to some $a\in N$ and slowly oscillating with respect to logarithmic summability, then $(u_k)$ is convergent to $a$.
\end{theorem}
Also, in view of theorem above and Theorem \ref{lboundedtoslowly} we get following theorem.
\begin{theorem}
Let sequence $(u_k)$ be in $(N,\mu,\nu)$. If $(u_k)$ is statistically $(L,m)$ summable to some $a\in N$ and $\{k\ln k(u_{k}-u_{k-1})\}$ is q-bounded, then $(u_k)$ is convergent to $a$.
\end{theorem}


\begin{thebibliography}{99}
\bibitem{zadeh} L. A. Zadeh,  Fuzzy sets, {\it Inf. Control} {\bf8} (1965), 338--353.
\bibitem{atanassov1983} K. Atanassov, Intuitionistic fuzzy sets, In: VII ITKR's Session, Sofia, June 1983 (Deposed in Central Sci.-Techn. Library of Bulg. Acad. of Sci., 1697/84) (in Bulgarian). Reprinted: International Journal of Bioautomation 2016; 20(S1): S1-S6 (in English).
\bibitem{atanassov} K. Atanassov,  Intuitionistic fuzzy sets, {\it Fuzzy Sets Syst.} {\bf 20} (1986), 87--96.
\bibitem{ifmetric} J.H. Park,  Intuitionistic fuzzy metric spaces, {\it Chaos Solitons Fractals} {\bf22} (2004), 1039--1046.
\bibitem{ifnorm} R. Saadati, J.H. Park, On the intuitionistic fuzzy topological spaces, {\it Chaos Solitons Fractals} {\bf 27} (2006), 331--344.
\bibitem{karakus} S. Karaku\c{s}, K. Demirci, O. Duman,  Statistical convergence on intuitionistic fuzzy normed spaces, {\it Chaos Solitons Fractals} {\bf 35} (2008), 763--769
\bibitem{mursaleen1} M. Mursaleen, S.A. Mohiuddine, On lacunary statistical convergence with respect to the intuitionistic fuzzy normed space, {\it J. Comput. Appl. Math.} {\bf 233} (2009), 142--149
\bibitem{muhiddin} S.A. Mohiuddine, Q.M. Danish Lohani, On generalized statistical convergence in intuitionistic fuzzy normed space, {\it  Chaos Solitons Fractals} {\bf 42} (2009), 1731--1737
\bibitem{mursaleen3} M. Mursaleen, S.A. Mohiuddine, O.H.H. Edely,  On the ideal convergence of double sequences in intuitionistic fuzzy normed spaces, {\it Comput. Math. Appl.} {\bf 59} (2010), 603--611
\bibitem{jifs1} A. Esi, B. Hazarika, $\lambda$-ideal convergence in intuitionistic fuzzy 2-normed linear space, {\it J. Intell. Fuzzy Syst.} {\bf24}(4) (2013), 725--732
\bibitem{jifs2} V. Karakaya, N. \c{S}im\c{s}ek, F. G\"{u}rsoy, M. Ert\"{u}rk,  Lacunary statistical convergence of sequences of functions in intuitionistic fuzzy normed space, {\it J. Intell. Fuzzy Syst.} {\bf26}(3) (2014), 1289--1299.
\bibitem{jifs3} E. Sava\c{s}, M. G\"{u}rdal, Certain summability methods in intuitionistic fuzzy normed spaces, {\it J. Intell. Fuzzy Syst.} {\bf27}(4) (2014), 1621--1629.
\bibitem{jifs4} P. Debnath, Results on lacunary difference ideal convergence in intuitionistic fuzzy normed linear spaces, {\it J. Intell. Fuzzy Syst.} {\bf28}(3) (2015), 1299--1306.
\bibitem{selma} S. Altunda\u{g}, E. Kamber, Weighted statistical convergence in intuitionistic fuzzy normed linear spaces, {\it J. Inequal. Spec. Funct.} {\bf8} (2017), 113--124.
\bibitem{jifs5} M. Kiri\c{s}ci, Fibonacci statistical convergence on intuitionistic fuzzy normed spaces, {\it J. Intell. Fuzzy Syst.} {\bf36}(6) (2019), 5597--5604.
\bibitem{yavuzcesaro} \"{O}. Talo, E. Yavuz, Ces\`{a}ro summability of sequences in intuitionistic fuzzy normed spaces and related Tauberian theorems, {\it Soft Comput.} {\bf25} (2021), 2315--2323
\bibitem{yavuzlog} E. Yavuz,  On the logarithmic summability of sequences in intuitionistic fuzzy normed spaces, {\it FUJMA} {\bf3}(2) (2020), 101--108.
\bibitem{lael} F. Lael, K. Nourouzi, Some results on the $IF-$normed spaces, {\it Chaos Solitons Fractals} {\bf37} (2008), 931--939
\bibitem{qbounded} H. Efe, C. Alaca, Compact and bounded sets in intuitionistic fuzzy metric spaces, {\it Demonstr. Math.} {\bf40}(2) (2007), 449--456.
\bibitem{nuray1} F. Nuray, U. Ulusu, E. D\"{u}ndar, Ces\`{a}ro summability of double sequences of sets. {\it Gen. Math. Notes} {\bf25}(1) (2014), 8--18.
\bibitem{nuray2} U. Ulusu, F. Nuray, Lacunary statistical summability of sequences of sets. {\it Konuralp J. Math.}, {\bf3}(2) (2015), 176--184.
\bibitem{nuray3} U. Ulusu, E. D\"{u}ndar, E. G\"{u}lle, $\mathcal{I}_2$-Ces\`{a}ro summability of double sequences of sets. {\it PJM}, {\bf9}(1) (2020), 561--568.
\bibitem{moriczcesaro} F. M$\acute{o}$ricz, Ordinary convergence follows from statistical summability $(C,1)$ in the case of slowly decreasing or oscillating sequences, {\it Colloq. Math.} {\bf99}(2) (2004), 207--219.
\bibitem{moriczlog} F. M$\acute{o}$ricz, Theorems relating to statistical harmonic summability and ordinary convergence of slowly decreasing or oscillating sequences, {\it Analysis} {\bf24} (2004), 127--145.
\end{thebibliography}
\end{document}